\documentclass[12pt]{amsart}

\usepackage[a4paper,top=2cm,bottom=2cm,left=3cm,right=3cm,marginparwidth=1.75cm]{geometry}

\usepackage{pgf,tikz}

\usepackage[breaklinks=true,colorlinks=true,linkcolor=blue,citecolor=red,urlcolor=blue,psdextra,pdfencoding=auto]{hyperref}

\newtheorem*{theorem*}{Theorem}

\newtheorem*{corollary*}{Corollary}
\newtheorem*{conj*}{Conjecture}

\newtheorem{proposition}{Proposition}
\newtheorem*{prop*}{Proposition}

\newtheorem*{assumption*}{Assumption}

\theoremstyle{definition}

\newtheorem*{definition*}{Definition}

\theoremstyle{remark}
\newtheorem{remark}{Remark}

\newtheorem*{notation*}{Notation}
\newtheorem*{algorithm*}{Algorithm}
\newtheorem*{example*}{Example}

\makeatletter
\@namedef{subjclassname@2020}{%
  \textup{2020} Mathematics Subject Classification}
\makeatother

\title{On triangular biregular degree sequences}

\author{Benjamin Egan}
\address{Department of Mathematical and Physical Sciences, La Trobe University, VIC 3086, Australia}
\email{20739919@students.edu.au}
\thanks{The first named author thanks the Department of Mathematical and Physical Sciences, La Trobe University, for supporting this project.}

\author{Yuri Nikolayevsky}
\address{Department of Mathematical and Physical Sciences, La Trobe University, VIC 3086, Australia}
\email{Y.Nikolayevsky@latrobe.edu.au}



\begin{document}

\begin{abstract}
A simple graph is called triangular if every edge of it belongs to a triangle. We conjecture that any graphical degree sequence all terms of which are greater than or equal to $4$ has a triangular realisation, and establish this conjecture for a class of biregular graphical degree sequences.
\end{abstract}

\maketitle

\section{Introduction}
\label{s:intro}

All graphs in this paper are finite and simple (no loops, no multiple edges). Given a graph $G$, we denote $E(G)$ and $V(G)$ its edge set and its vertex set, respectively. We call a sequence $d=(d_1, d_2, \dots, d_n)$ of nonnegative integers, where $n =|V(G)|$, the degree sequence of the graph $G$, if there is a labelling of the vertices $\{v_1, v_2, \dots, v_n\} = V(G)$ such that $\deg(v_i)=d_i$, for $i=1, 2, \dots, n$. Unless otherwise is stated explicitly, we will always order the terms of a degree sequence in a non-increasing order. A sequence of $p \ge 0$ repeated terms $a \ge 0$ is denoted $a^p$. A sequence $d$ of non-negative integers is called \emph{graphic} if there exists a graph $G$ whose degree sequence is $d$; any graph with this property is called a \emph{realisation} of $d$. There is a vast literature on graphical degree sequences, starting from the pioneering paper of Erd\H{o}s-Gallai \cite{EG}.

In this paper, we are interested in triangular realisations of degree sequences. We call a graph \emph{triangular}, if every edge of it belongs to a $3$-cycle (for results on triangular graphs in the context of degree sequences, see \cite{PR}; for the other extremal case, the degree sequences of triangle-free graphs we refer the reader to~\cite{EFS}). It is easy to see that regular graphical sequences $(1^n)$ have no triangular realisation, and the sequences $(2^n)$ and $(3^n)$ admit a triangular realisation if and only if $n$ is divisible by $3$ and by $4$, respectively, as disjoint unions of $K_3$'s and $K_4$'s, respectively (we do not require realisations to be connected). We found no counterexamples to the following conjecture.

\begin{conj*}
Any graphical sequence all of whose terms are greater than or equal to $4$ admits a triangular realisation. 
\end{conj*}

We establish the conjecture for regular and biregular sequences (the sequences whose terms take no more than two values):

\begin{theorem*}
  Any graphical sequence of the form $(a^p, b^q)$, where $a>b\ge 4$ and $p \ge 0$, $q > 0$, admits a triangular realisation.
\end{theorem*}

The author would like to express their deepest gratitude to Grant Cairns from conversation with whom came the Conjecture, and for his many useful comments and suggestions.


\section{Preliminaries}
\label{s:2v}

It is well known that a constant sequence $(a^n)$, with $a,n > 0$, is graphical if and only if $na$ is even and $a \le n-1$.

The Erd\H{o}s-Gallai conditions for graphicality of a biregular sequence $(a^p,b^q)$, with $a>b > 0$ and $p,q >0$, take especially nice form (\cite{TV} or explicitly in \cite[Theorem~5]{CMN}): such a sequence is graphical if and only if $ap+bq$ is even and the following inequalities are satisfied:
\begin{equation}\label{eq:EG2}
  a \le p+q-1,\qquad ap \le p(p-1)+bq.
\end{equation}
It is not hard to see that when $b \ge p$ or when $b=a-1$, the second inequality in \eqref{eq:EG2} follows from the first one.

We will systematically use the \emph{Havel-Hakimi reduction} \cite{Hav, Hak}: given a degree sequence $d=(d_1, d_2, \dots, d_n)$ and a number $1 \le k \le n$, we form the degree sequence $d'$ by removing the term $d_k$ from $d$ and then reducing the first $k$ of the remaining terms by $1$. The resulting sequence $d'$ is graphical if and only if $d$ is graphical, and moreover, if it is, there exists a realisation of $d$ in which a vertex of degree $d_k$ is connected to the vertices having the top $k$ degrees in $d$, not counting itself \cite[Theorem~2.1]{KW}.

\section{Proof of the Theorem}
\label{s:proof}

The proof goes as follows. We first establish the Theorem for regular sequences (Proposition~\ref{p:sb1}) by explicitly constructing the corresponding triangular realisations as certain circulant graphs. In the same proposition, we construct triangular realisations for graphical sequences whose terms differ by no more than $1$. To be able to proceed to the general case, we need some extra structure for these realisations, namely we require that they contain so called top monotone cycles (defined below). We then consider arbitrary biregular graphical sequences $(a^p,b^q)$. By repeated use of the Havel-Hakimi reduction, we reduce such a sequence to a sequence whose terms differ by no more than $1$, take its triangular realisation provided by Proposition~\ref{p:sb1} and then running the Havel-Hakimi reduction in reverse, re-attach new vertices to the vertices of the top monotone cycle, ``in the correct order". The proof splits into two cases: in Proposition~\ref{p:pgeb} we work with the case $p \ge b$, and in Proposition~\ref{p:pltb}, with the case $p < b$.

Throughout the proof, we call a cycle (respectively, a path) of a graph $G$ a subset of $G$ homeomorphic to a circle (respectively, to a closed interval or a single point).

\medskip

\subsection{Regular sequences and sequences of width 1}
\label{ss:const}

In this section we consider graphical sequences which are either regular or whose maximal and minimal terms differ by one.

Provided all terms of such a sequence are at least $4$, we construct a triangular realisation with an additional property which we will use in the next section.

Given a graph $G$ with the degree sequence $d=(d_1, d_2, \dots, d_n)$ (recall that the terms of $d$ are ordered in a non-increasing order) and a number $3 \le m \le n$, we call a cycle $C_m \subset G$ a \emph{top monotone $m$-cycle}, if we can label its vertices $v_1, v_2, \dots, v_m$ in the consecutive order in such a way that $\deg(v_i)=d_i$ for $i=1, 2, \dots, m$.

\begin{proposition} \label{p:sb1}
  Let $d=((b+1)^p, b^q)$ be a graphical sequence such that $q > 0, \, p \ge 0$ and $b \ge 4$. Then the sequence $d$ has a triangular realisation. Moreover, for any $m \ge \max(3, p), \, m \le p+q$, with the exception of the case when $(p,q,m)=(b,2,b+2)$ \emph{(}which corresponds to the complete graph $K_{b+2}$ minus an edge\emph{)}, such triangular realisation $G$ can be chosen to contain a top monotone $m$-cycle.
\end{proposition}

\begin{remark} \label{r:Kminuse}
  In the exceptional case $(p,q,m)=(b,2,b+2)$ in Proposition~\ref{p:sb1}, we have $d=((b+1)^b, b^2)$. The only graph $G$ realising this degree sequence is the complete graph $K_{b+2}$ minus an edge. The graph $G$ is clearly triangular (for $b \ge 4$), but does not contain a top monotone (Hamiltonian) $(b+2)$-cycle, as its two vertices of degree $b$ are not connected by an edge. However, for any $2 \le r \le \frac12b+1$, it contains a $(b+2)$-cycle $C'_r$ in which these two vertices are at distance $r$; we will use this cycle later in the proof of Proposition~\ref{p:pgeb}.
\end{remark}

In the first paragraph of the proof, we construct a triangular realisation of the regular graphical sequence $(b^q)$, where $b \ge 4$ is even. For the rest of the proof we refer to that realisation as to the \emph{standard realisation of $(b^q)$}.

\begin{proof}
  We first consider a regular graphical sequence $d=(b^q)$. Note that $q \ge b+1 \ge 5$. We construct a triangular realisation of $d$ as the following circulant graph. Take $V=\mathbb{Z}_q$, which we identify with the set $\{0,1,\dots,q-1\}$, and join the vertices $i,j \in V$ by an edge if and only if $i-j \in S$, where the subset $S \subset \mathbb{Z}_q$ of cardinality $b$ has the following three properties: $0 \notin S, \; -S = S$ and $S+S \supset S$. It is not hard to see that any choice of such $S$ gives a triangular realisation of $d$. We will use the following choices of $S$. If $b=2c$ is even (note that $c \ge 2$), define $S=\{\pm1, \pm 2, \dots, \pm c\}$. If $b=2c+1$ is odd (again, $c \ge 2$), then $q=2l$ is even, and we take $S=\{\pm 1, l, l \pm 1, \dots, l \pm (c-1)\}$.

  The triangular graph $G$ so constructed contains a top monotone cycle $C_m \subset G$, for any $3 \le m \le q$. To see this, we first suppose that $b=2c$ is even. In the case when $m=2s$ is even ($s \ge 2$), we take $C_m=(0, 2, \dots, 2s-2, 2s-1, \dots, 3, 1)$, and when $m=2s+1$ is odd, take $C_m=(0, 2, \dots, 2s, 2s-1, \dots, 3, 1)$. Now let $b=2c+1$ be odd; then $q=2l$ is even and $l \ge 3$. If $m=2s$ is even, take $C_m=(0, 1, \dots, s-1, s+l-1, s+l-2, \dots, l+1, l)$, and if $m=2s+1$ is odd, we can take $C_m=(0, 1, \dots, s-1, s+l-1, s+l-2, \dots, l+1)$.

  \medskip

  Now let $d=((b+1)^p,b^q)$, with $b \ge 4,\, p,q > 0$ be a graphical sequence (recall that this is equivalent to the fact that $(b+1)p + bq$ is even and $b+2 \le p+q$ --- see Section~\ref{s:2v}). If $p \ge b+2$, we can apply the Havel-Hakimi reduction to $d$ on a vertex of degree $(b+1)$ to obtain the graphical degree sequence $d'=((b+1)^{p-b-2},b^{q+b+1})$. As $q+b+1 > 2$, this sequence is not one of the excluded sequences in the proposition. Suppose for a given $\max(3,p) \le m \le p+q$, we can construct a required triangular realisation $G'$ containing a top monotone cycle $C_{m-1}=(v_1, \dots, v_{m-1})$ with $\deg(v_i)=b+1$ for $i \le p-b-2$ and $\deg(v_i)=b$ for $i > p-b-2$ (note that $p > 3$, and so $\max(3,p)=p$ which implies $m-1 \ge \max(3,p-b-2)$; and clearly, $m-1 \le p+q-1$). To construct $G$, we add to $G'$ a vertex $u$ and join it with the vertices $v_{p-b-1}, \dots, v_{p-1}$ of $C_{m-1} \subset G$ (recall that $m \ge p$). The resulting graph $G$ is clearly triangular, as all the added edges $(u,v_i)$ belong to at least one of the $3$-cycles $(u,v_i,v_{i\pm1})$. Moreover, the cycle $C_m=(v_1, \dots, v_{p-b-2}, v_{p-b-1},u, v_{p-b}, \dots, v_{p-1},v_p,\dots,v_{m-1}) \subset G$ is a top monotone $m$-cycle. Repeating this procedure if necessary we see that it suffices to prove the claim under additional assumption
  \begin{equation}\label{eq:Cm}
  1 \le p \le b+1.
  \end{equation}

  We consider two cases, according to whether $b$ is odd or even, and we divide each case into several subcases.

  \underline{Case 1.} Suppose $b$ is odd. Then $b \ge 5$ and $q$ is even. Denote $c=\frac12(b+1), \; c \ge 3$.

  We first consider a particular case when $q=2$ and $m=p+2$ (so that we require $C_m$ to be Hamiltonian). Then $d=((2c)^p,(2c-1)^2)$, and the graphicality condition gives $2c \le p+1$. We assume that $2c \le p$ to avoid the case excluded in the statement (so from \ref{eq:Cm} we have $p=b+1$, but we do not need this fact). Construct the standard realisation $G'$ of the sequence $((2c)^{p+2})$ as above: take $V=\mathbb{Z}_{p+2}$ as the vertex set of $G'$, and join the vertices $i,j \in V$ by an edge if and only if $i-j \in \{\pm1, \pm2, \dots, \pm c\}$. Remove from $G'$ the edges $(0,c)$ and $(1,c+1)$ (this does not violate triangularity) and add the edge $(0,c+1)$ (note that $c+1 \ne -c$ in $\mathbb{Z}_{p+2}$, as $2c \le p$). The resulting graph $G$ realises our degree sequence $d$ and is triangular, as the edge $(1,c+1)$ belongs to the $3$-cycle $(0, 2,c+1)$. The vertices of degree $(2c-1)$ of $G$ are $1$ and $c$; all the other vertices have degree $2c$. The required cycle $C_m$ is given by $(1,c,c+1,c+2,\dots, -1,0,c-1, c-2, \dots, 2)$.

  In what follows we assume that either $q > 2$ or $m < p+q$.

  Suppose $b \ge 7$. Denote $c=\frac12(b+1)$ and construct the standard realisation $G'$ of the sequence $((b+1)^{p+q})$: take $V=\mathbb{Z}_{p+q}$ as the vertex set of $G'$, and join the vertices $i,j \in V$ by an edge if and only if $i-j \in \{\pm1, \pm2, \dots, \pm c\}$. Denote $r = \frac12 q$ and remove from $G'$ the edges $(0,1), \, (2,3), \dots, (2r-2,2r-1)$. The resulting graph $G$ realises the degree sequence $d=((b+1)^p,b^q)$. Moreover, $G$ is triangular. Indeed, any edge $(i,i+1)$ of $G$ belongs to the $3$-cycle $(i,i+1,i+3)$ (note that $c \ge 4$ as $b \ge 7$), and all the other edges are covered by the $3$-cycles $(i, i+c_1,i+c_1+c_2)$ (where $c_1,c_2 \ge 2, \, c_1+c_2 \le c$), except for the edges $(i,i+3)$ when $c=4$. But in the latter case, $G$ contains at least one of the edges $(i+3,i+2), \, (i+3,i+4)$, and so the edge $(i,i+3)$ belongs to the corresponding $3$-cycle (note that in the graph $G$, the corresponding degree sequence $d_i$ for $i\in\mathbb{Z}_{p+q}$ is non-decreasing, contrary to our usual convention).

  We next show that for all $\max(3,p) \le m \le p+q$, the graph $G$ so constructed contains a top monotone $m$-cycle $C_m$ (that is, a cycle $C_m$ containing a path $P$ passing through all the vertices of degree $(b+1)$ of $G$). We start with defining the path $P$. If $p=1$, the path $P$ consists of a single vertex $-1$. Otherwise, we take for $P$ the union of the edge $(q,q+1)$ and the paths $(q,q+2,q+4, \dots, q_1)$ and $(q+1,q+3,q+5, \dots, q_2)$, where $\{q_1,q_2\}=\{-1,-2\}$. The path $P$ so constructed has endpoints $-1$ and $-2$ and passes through all the vertices of $G$ of degree $(b+1)$. Let $m'=m-p$ (note that $0 \le m' \le q\, (=2r)$). Suppose $m'$ is odd (so that $m'=2l+1$ with $0 \le l \le r-1$). If $P$ contains at least two vertices, we attach to $P$ two paths $P_1=(-1,1,3, \dots, 2l-1)$ and $P_2=(-2,0,2, \dots, 2l)$ and the edge $(2l-1,2l)$ (note that this edge lies in $G$ by construction) which produces a required cycle $C_m$. If $P=\{-1\}$, then $p=1$ and so $m' \ge 2$; then we repeat the above construction, with the first vertex of $P_2$ replaced by $-1$. Now suppose $m'=2l$ is even, $0 \le l \le r$. If $l=0$ or $l=1$, we can assume that $P$ contains at least two vertices (if $l=0$, we have $p=m \ge 3$, and if $l=1$ and $p=1$, we have $m=3$ and we can take for $C_3$ any triangle containing the vertex $-1$). Then we add to $P$ the edge $(-2,-1)$ when $l=0$, and the path $(-2,1,2,-1)$ when $l=1$ (this is always possible unless the vertex $2$ has degree $b+1$, that is, when $q=2$, but then $m=p+2$, which brings us to a sequence considered at the start of this case). If $l \ge 2$ and $P$ contains at least two vertices, we attach to $P$ two paths $P_1=(-1,1,3, \dots, 2l-5)$ and $P_2=(-2,0,2, \dots, 2l-4)$ and the path $(2l-4,2l-1,2l-3,2l-2,2l-5)$ (note that all its edges lie in $G$ by construction). If $l \ge 2$ and $P=\{-1\}$, the same construction works, with the first vertex of $P_2$ replaced by $-1$.

  Suppose $b=5$. We have $d=(6^p,5^{2r})$, where $r \ge 1, \, p+2r \ge 7$ and where by~\eqref{eq:Cm} we can assume that $1 \le p \le 6$. When $r=1$, we can additionally assume that $m < p+2$ (the case $m=p+2$ is excluded in the statement).

  We start with a graph $G'_N$ which gives a triangular realisation of the sequence $(5^{2N})$ (where $N \ge 3$ will depend on $r$ and $p$) as constructed above, but will present it a little differently. Namely, we consider a disjoint union of two $N$-cycles $C^1$ and $C^2$ and label the vertices of $C^\alpha, \, \alpha=1,2$, in the consecutive order, by $i_\alpha$, where $i \in \mathbb{Z}_N$. We then add all the edges of the form $(i_1,i_2)$ and $(i_1, (i\pm1)_2)$. The resulting graph $G'_N$ is clearly triangular and realises the sequence $(5^{2N})$. Note that removing any number of edges $(i_1,(i+1)_2)$ from $G'_N$ does not violate triangularity.

  We first consider the case when $p$ is odd. When $p=1$ we have $d=(6^1,5^{2r})$, where $r \ge 3$. We take the graph $G'_r$, remove the edges $(0_1,1_2), \, (1_1,2_2), (2_1,3_2)$, add a vertex $u$ and join it to the all six endpoints of these removed edges of $G'_r$. The resulting graph $G$ realises $d$ and is triangular. To construct a top monotone cycle $C_m \subset G$, we define $l=\lfloor\frac12m\rfloor$ and take the union of the path $(1_2,u,1_1)$ and the path $(1_1, 2_1, \dots, l_1, (m-l-1)_2, \dots, 1_2)$. Similarly, for $p=3$, our sequence is given by $d=(6^3,5^{2r})$, where $r \ge 2$. We take the graph $G'_{r+1}$, remove the edges $(0_1,1_2)$ and $(1_1,2_2)$, add a vertex $u$ and join it to the four endpoints of these removed edges and to the vertices $2_1$ and $3_2$ of $G'_{r+1}$. We obtain a graph $G$ which realises $d$ and is triangular (its vertices of degree $6$ are $u, 2_1$ and $3_2$). We construct a top monotone cycle $C_m \subset G$ as follows. If $m=2l, \, 2 \le l \le r+1$, the cycle $C_m$ is given by $(u,2_1, \dots, l_1, (l+1)_1,(l+1)_2, l_2, \dots, 3_2)$. If $m=2l+1$, we have $1 \le l \le r+1$. When $l \le r-1$, the cycle $C_m$ is given by $(u,2_1, \dots, l_1, (l+1)_1,(l+2)_2, (l+1)_2, \dots, 3_2)$; for $l=r+1$ we take $C_{2r+3}=(u,2_1, 1_1, 0_1 \dots, 3_1, 2_2,1_2,0_2, \dots, 3_2)$; for $l=r$ we take $C_{2r+1}=(u,2_1, 3_1, \dots, 0_1, 1_1, 0_2, -1_2, \dots, 3_2)$. Finally, let $p=5$. Then $d=(6^5,5^{2r})$, where $r \ge 1$. If $r=1$, the only realisation of $d$ is $K_7$ minus an edge which is triangular and has top monotone cycles $C_m$ for $m=5,6$ (case $m=7$ is excluded in the statement of the proposition). We can now assume that $r \ge 2$. Take the graph $G'_{r+2}$, remove the edge $(2_1,2_2)$, add a vertex $u$ and join it to the vertices $i_1, i_2, \; i=0,1,2$. The resulting graph $G$ realises $d$ and is triangular (its vertices of degree $6$ are $0_1, 1_1,0_2, 1_2$ and $u$). A top monotone cycle $C_m \subset G$ for $m \ne 7$ is given by $(u,0_1,1_1, \dots, (l-1)_1,(m-l-2)_2, (m-l-2)_2, \dots, 1_2,0_2)$, where $l=\lfloor\frac12m\rfloor$, and by $C_7=(u,1_1,0_1,-1_1,-1_2,0_2,1_2)$.

  Now suppose that $p$ is even. Let $p=2s,\, s=1,2,3$. Similar to the above, we start with the triangular graph $G'_{r+s}$ realising the sequence $(5^{2s+2r})$ (note that $s+r \ge 4$). If $s=1$, we add to $G'_{r+1}$ the edge $(0_1,-2_2)$. To construct a top monotone cycle $C_m$ in the resulting graph $G$, we take the union of the edge $(0_1,-2_2)$, the paths $(0_1, 1_1, \dots, l_1)$ and $(-2_2, -1_2,0_2, \dots, t_2)$, where $l-t \in \{0,1\}$ and the edge $(l_1,t'_2)$. This construction works unless $m=p+q \, (=2r+2)$, in which case we take $C_m=(0_1,1_1, \dots, -1_1,-1_2,0_2,1_2 \dots,-2_2)$. When $s=2$, we add to $G'_{r+2}$ the edges $(0_1,-2_2)$ and $(1_1,-1_2)$. The top monotone cycles $C_m$ for $m < p+q$ are the same as in the case $s=1$. For $m=p+q$, the cycle $C_m$ constructed for $s=1$ does not work, as it does not pass through all the vertices of degree $b+1$ consecutively: we take $C_m=(0_1,1_1, \dots, -1_1,0_2,1_2, \dots,-2_2,-1_2)$ instead. For $s=3$ and $m <p+q$, the similar construction works: the graph $G$ is obtained by adding edges $(0_1,-2_2),(1_1,-1_2)$ and $(2_1,0_2)$ to $G'_{r+3}$; the top monotone cycles are the same as in the previous two cases. When $s=3$ and $m=p+q$ (that is, for the sequence $d=(6^6,5^{2r}), \, r \ge 1$, when we need a Hamiltonian top monotone cycle), our construction is different. When $r \ge 4$, take $G'_{r+3}$ and add to it the edges $(0_1,2_1), (1_1,3_1)$ and $(4_1,6_1)$. The cycle $C_{2r+6}$ is given by the union of the path $(6_1,4_1,3_1,2_1,1_1,0_1,0_2,1_2,2_2,3_2,4_2,5_2,5_1,6_2)$ and the path $(6_1,7_1, \dots,-1_1,-1_2, -2_2, \dots, 6_2)$. The remaining cases are $r \in \{1,2,3\}$, which gives the sequences $d=(6^6,5^2), \, (6^6,5^4)$ and $(6^6,5^6)$. For $d=(6^6,5^6)$ we take the standard realisation of the sequence $(6^{12})$ and remove the edges $(0,3), (1, 4)$ and $(2,5)$, which gives a triangular graph, with a Hamiltonian top monotone cycle $C_{12}=(0,1,2,\dots, 11)$. Similarly, for $d=(6^6,5^4)$ we take the standard realisation of the sequence $(6^{10})$ and remove the edges $(0,3)$ and $(2,5)$. The resulting graph is triangular, with a Hamiltonian top monotone cycle $C_{10}=(4,6,7,8,9,1,0,2,3,5)$. And finally, for $d=(6^6,5^2)$ and we take for our graph $G$ the complete graph $K_8$ minus the disjoint union of two $2$-paths and an edge. It is clearly triangular and a Hamiltonian top monotone cycle $C_{8}$ can be easily constructed.

  \underline{Case 2.} Now suppose $b$ is even. Then $p$ is also even. Denote $b=2c$ and $p=2r$, so that $d=((2c+1)^{2r},(2c)^q)$. We have $c \ge 2$ and $2r+q \ge 2c+2$. Also, by \eqref{eq:Cm} we can assume that $1 \le r \le c$. From this and the previous inequality we have $q \ge 2$. If $q=2$, then $r=c$, and (the only) graph $G$ realising $d$ is $K_{2c+2}$ minus an edge. It is easy to see that such $G$ is triangular, and that for any $2c \le m \le 2c+1$ there is a top monotone cycle. If $m=2c+2$ we come to the excluded case (see Remark~\ref{r:Kminuse}). We will therefore assume that $q \ge 3$.  

  We start with the standard realisation $G'$ of the sequence $((2c)^{2r+q})$ with the vertex set $\mathbb{Z}_{2r+q}$ and with two vertices $i, j \in \mathbb{Z}_{2r+q}$ being connected by an edge if and only if $i-j \in \{\pm 1, \pm 2, \dots, \pm c\}$. We construct the required graph $G$ by adding $r$ edges of the form $(i,i+c+1)$ to $G'$. The resulting graph will always be triangular, and we will only need to establish the top monotone cycle property for it.

  First let $r=1$. Then $G$ is obtained from $G'$ by attaching the edge $(0,c+1)$. For $3 \le m \le c+2$, a cycle $C_m$ is given by $(0,1, \dots, m-2,c+1)$. If $m > c+2$, we denote $m'=m-c-2, \, l=\lfloor m'/2 \rfloor$, and construct the cycle $C_m$ as the union of the $(c+1)$-path $(1, \dots, c,c+1,0)$, two paths $(0, -2, -4, \dots, -2l)$ and $(1, -1, -3, \dots, 1-2(m'-l))$, and the edge $(-2l,1-2(m'-l))$.

  Next suppose $m=2r+q$ (so that the required cycle $C_m$ is Hamiltonian) and $r=c$. We construct $G$ by adding to $G'$ the edges $(0, c+1), (2, c+3), (3,c+4), \dots, (c, 2c+1)$. Then $C_m=(0,c+1,c+3, c+4, \dots, 2c+1, c, c-1, \dots, 3, 2, c+2,2c+2, 2c+3, 2c+4, \dots, 2c+q-1, 1)$ (note that $2c+q-1 = -1$ in $\mathbb{Z}_{2c+q}$).

  In all the other cases ($r \ge 2$ and either $m < 2r+q$ or $r \le c-1$) we take for $G$ the graph obtained from $G'$ by attaching the edges $(0, c+1), (1, c+2), \dots, (r-1, r+c)$. Note that the vertices $0, 1, \dots, r-1$ and $c+1, \dots, c+r$ of $G$ have degree $(2c+1)$. The vertices of degree $2c$ form two ``connected components": $r, \dots, c$ of cardinality $c-r+1$ and $c+r+1, c+r+2, \dots, -2,-1$ of cardinality $q+r-c-1$. First suppose that $m \le 3r+q-c-1$ (so that in $C_m$, there are no more vertices of degree $2c$ than in the second ``component"); note that if $r=c$, then the inequality $m \le 3r+q-c-1$ is always satisfied (as we assume $m < 2r+q$). We denote $m'=m-2r, \, l=\lfloor m'/2 \rfloor$, and construct the cycle $C_m$ as the union of the $(2r-1)$-path $(0,c+1,c+2, \dots, c+r, r-1, r-2, \dots, 1)$, two paths $(0, -2, -4, \dots, -2l)$ and $(1, -1, -3, \dots, 1-2(m'-l))$, and the edge $(-2l,1-2(m'-l))$. Finally let $m > 3r+q-c-1$ (and we can assume that $r \le c-1$). Denote $m'=m-(3r+q-c-1)$ and construct the cycle $C_m$ as the union of the $(2r-1)$-path $(0,c+1,c+r, c+r-1, \dots, c+2, 1, 2, \dots, r-1)$ and the path $(r-1,c-m'+1,c-m'+2,\dots, c-1, c, r+c+1, r+c+2, \dots, -2,-1, 0)$ (note that degree $2c$ vertices $c$ and $r+c+1$ are connected by an edge as $r+1 \le c$, and that $c-m'+1 \ge r$).
\end{proof}

\subsection{Biregular sequences \texorpdfstring{$(a^p,b^q)$}{(a\unichar{"005E}p,b\unichar{"005E}q)} with \ensuremath{p \texorpdfstring{\ge}{\unichar{"2265}} b}}
\label{ss:pgeb}

Let $d=(a^p,b^q)$ with $a>b \ge 4$ and $p,q >0$ be a graphical sequence. We are aiming to construct its triangular realisation. By Proposition~\ref{p:sb1}, we can additionally assume that $a \ge b+2$. In this section, we consider the case $p \ge b$.

\begin{proposition} \label{p:pgeb}
  Let $d=(a^p,b^q)$, with $a\ge b+2, \, b \ge 4$ and $p,q >0$, be a graphical sequence. Suppose $p \ge b$. Then the sequence $d$ admits a triangular realisation.
\end{proposition}

\begin{proof}
  We start by applying the Havel-Hakimi reduction (see Section~\ref{s:2v}) to the sequence $d(0):=d$ starting from a vertex of degree $b$. As $b \le p$, at the first step, we remove a vertex of degree $b$, reduce by $1$ the degrees of $b$ out of the first $p$ vertices and rearrange the degrees in the non-increasing order. We obtain a (still graphical) sequence $d(1)=(a^{p-b},(a-1)^b,b^{q-1})$. If $d(1)_1 > b+1$ and $q-1 >0$ we apply the Havel-Hakimi reduction to the sequence $d(1)$, again starting from a vertex of degree $b$. We get a sequence $d(2)$ given either by $(a^{p-2b},(a-1)^{2b},b^{q-2})$ if $p \ge 2b$, or by $((a-1)^{2p-2b},(a-2)^{2b-p},b^{q-2})$ if $p < 2b$. If $d(2)_1 > b+1$ and if $q-2 > 0$, we apply the Havel-Hakimi reduction again to obtain a sequence $d(3)$, and so on. We continue until we get a sequence $d(l)$, where $l \ge 1$ is the smallest number with one of the two properties: either $l=q$ (that is, we have ``run out of vertices of degree $b$''), or $l < q$ and $d(l)_1 = b+1$. It is easy to see that all the sequences $d(j), \, j=0, 1, \dots, l$, have at least $p$ terms and satisfy the inequality
  \begin{equation}\label{eq:dk1p}
  d(j)_1-d(j)_p \le 1.
  \end{equation}
  Moreover, by our construction, the sequence $d(l)$ satisfies the conditions of Proposition~\ref{p:sb1}, and the last term of the sequence $d(l)$ is greater than or equal to $b$. By  Proposition~\ref{p:sb1}, we can construct a triangular graph $G_l$ whose degree sequence is $d(l)$ and which contains a top monotone $p$-cycle $C_p$. The only exception is when the degree sequence $d(l)$ has a unique realisation $G_l$ which is the complete graph minus an edge, and the length of $d(l)$ is $p$ (that is, $l=q$ and $d(l)=((p-1)^{p-2},(p-2)^2)$). In this exceptional case, there is no top monotone cycle of length $p$, and we take for the cycle $C_p$ a Hamiltonian cycle in $G_l$ in which one of the arcs between the two vertices of degree $p-2$ has length $b-1$ (such a cycle clearly exists; see Remark~\ref{r:Kminuse}). We now perform $l$ steps of the Havel-Hakimi reduction in reverse constructing the sequence of graphs $G_l \subset G_{l-1} \subset \dots \subset G_1 \subset G_0$ with the following properties: for all $k= l-1, \dots, 1, 0$,

  \begin{itemize}
    \item the degree sequence of $G_k$ is $d(k)$,
    \item every graph $G_k$ is triangular, and
    \item the cycle $C_p \subset G_k$ is a top monotone $p$-cycle in $G_k$.
  \end{itemize}

  Suppose that for some $k=l,l-1, \dots, 1$, a graph $G_k$ with these three properties is already constructed. Its degree sequence $d(k)$ has the form $d(k)=((c+1)^{p_1}, c^{p_2},b^{q'})$, where $q', p_1 \ge 0, \, p_2 > 0, \, p_1 + p_2 =p$, and $c \ge b$. Moreover, $G_k$ contains a top monotone cycle $C_p$ whose consecutive vertices $v_1, \dots, v_p$ satisfy $\deg_{G_k}(v_i)=c+1$ for $i=1, \dots, p_1$, and $\deg_{G_k}(v_i)=c$ for $i=p_1+1, \dots, p$ (note that we may have $c=b$).

  As the degree sequence $d(k-1)$ satisfies the inequality~\eqref{eq:dk1p}, we can have one of two cases (note that there may exist several different degree sequences ending with $b$ from which we obtain the sequence $d(k)$ by removing that final term $b$ and subtracting $1$ from the top $b$ terms, but among them, there is a \emph{unique} degree sequence, namely $d(k-1)$, whose first and $p$-th terms differ by no more than $1$): either $p_2 \ge b$ and then $d(k-1)=((c+1)^{p_1+b}, c^{p_2-b},b^{q'+1})$, or  $p_2 < b$ and then $d(k-1)=((c+2)^{b-p_2}, (c+1)^{p+p_2-b},b^{q'+1})$.

  We now construct the graph $G_{k-1} \supset G_k$ by adding a vertex $u$ of degree $b$ to $G_k$ and joining it to $b$ vertices of the cycle $C_p$ as follows. In the exceptional case (which may only occur once, when we pass from $G_l$ to $G_{l-1}$), we join the vertex $u$ to $b$ consecutive vertices of the cycle $C_p$ subtending the arc of $C_p$ whose endpoints have degree $p-2$. The resulting graph $G_{l-1}$ is triangular, has the degree sequence $d(l-1)$, and the same cycle $C_p \subset G_{l-1}$ becomes a top monotone $p$-cycle (in which we then label the vertices according to the above).

  In all the other cases, we have one of the two possibilities. In the first case ($p_2 \ge b$), the vertex $u$ is joined to the vertices $v_{p_1+1}, \dots, v_{p_1+b}$. In the second case ($p_2 < b$), we join the vertex $u$ to the vertices $v_{p_1+1}, \dots, v_p, v_1, \dots ,v_{b-p_2}$. In both cases, the resulting graph $G_{k-1}$ has the required degree sequence $d(k-1)$ and is triangular. Moreover, the same cycle $C_p \subset G_{k-1}$ is a top monotone cycle, with the same labelling of the vertices.
\end{proof}

\subsection{Biregular sequences \texorpdfstring{$(a^p,b^q)$}{(a\unichar{"005E}p,b\unichar{"005E}q)} with \texorpdfstring{$p < b$}{p < b}}
\label{ss:pltb}

Let $d=(a^p,b^q)$, where $a \ge b+2$, $b \ge 4$ and $p,q >0$, be a graphical sequence (recall that this is equivalent to the fact that $ap+bq$ is even and that $a \le p+q-1$, see Section~\ref{s:2v}). In this section, we consider the case $p < b$.

We start with some inequalities reducing the set of possible cases.

\begin{remark} \label{r:noKb+1}
In the above settings, we can assume that
\begin{equation}\label{eq:noKb+1}
a + b \ge p+q-1.
\end{equation}
Indeed, suppose $a+b < p+q-1$ and denote $q'=q-(b+1)$. We have $q'=q-(b+1) > a-p > b-p > 0$. Moreover, the number $ap + bq'$ is still even and $a \le p + q' -1$, so the sequence $(a^p, b^{q'})$ is graphical. If it has a triangular realisation $G$, then the disjoint union of $G$ and $K_{b+1}$ is a triangular realisation of the original sequence $(a^p,b^q)$. We can therefore keep reducing $q$ by $b+1$ until we reach the inequality~\eqref{eq:noKb+1}.
\end{remark}

\begin{remark} \label{r:2b}
Note that if $2b > p + q$, then \emph{any} realisation of the sequence $(a^p,b^q)$ is triangular. Indeed, if $G$ is a graph realising this sequence and there is an edge $(u,v)$ between two vertices of $G$, then $\deg(u) + \deg(v) > p+q$, and so there is a vertex of $G$ different from both $u$ and $v$ to which both $u$ and $v$ are connected (this argument, of course, applies to general sequences whose minimum is greater than half of the number of vertices). So we can assume that
\begin{equation}\label{eq:2b}
2b \le p+q.
\end{equation}
\end{remark}

We have the following.

\begin{proposition} \label{p:pltb}
  Let $d=(a^p,b^q)$ with $a \ge b+2, \, b \ge 4$ and $p,q >0$ be a graphical sequence. If $b \ge p+4$, the sequence $d$ admits a triangular realisation.
\end{proposition}

\begin{proof}
  Our triangular realisation will consist of the complete graph $K_p$ on the vertices of degree $a$, of a triangular (or at least, ``sufficiently connected") graph on the vertices of degree $b$, and of $p(a-p+1)$ edges in between these two graphs.

  We can assume the inequalities~\eqref{eq:noKb+1} and~\eqref{eq:2b} to be satisfied. Note that~\eqref{eq:2b} and the assumption $b \ge p+4$ imply $q \ge b+4$.

  Moreover, as $b \ge p+4$ and $q \ge a-p+1$ by~\eqref{eq:EG2} (and $a-p+1 > b-p+1 > 0$) we obtain $bq-p(a-p+1) \ge 4q$. Let $bq-p(a-p+1)= cq + q_1$, where $0 \le q_1 < q$, and denote $q_2 = q-q_1 > 0$. Then
  \begin{equation}\label{eq:ineqc4}
    b > c \ge 4.
  \end{equation}

  We claim that the sequence $d'=((c+1)^{q_1}, c^{q_2})$ is graphical. Indeed, its sum equals $q_1(c+1)+q_2 c = bq-p(a-p+1)$ which is even provided $ap+bq$ is. The first inequality of~\eqref{eq:EG2} (the only one we need to verify for graphicality) takes the form $c+1 \le q-1$ when $q_1 \ne 0$ and the form $c \le q-1$ when $q_1 = 0$. In both cases, the required inequality follows from~\eqref{eq:ineqc4} and the fact that $q \ge b+4$, as noticed above.

  By Proposition~\ref{p:sb1}, the sequence $d'$ admits a triangular realisation $G_2$ containing a top monotonic (Hamiltonian) cycle $C_q$ (note that from~\eqref{eq:ineqc4} and the fact that $q \ge b+4$ we get $c \le b-1 \le q-5$, so the excluded case in Proposition~\ref{p:sb1} may not occur). Let $C_q=(v_1, v_2, \dots, v_q)$, where $\deg v_i=c+1$ for $i=1, \dots,q_1$, and $\deg v_i=c$ for $i=q_1+1, \dots,q$. Take $G_1=K_{p}$ to be the complete graph. In the realisation $G$, the vertices of $G_1$ will have degree $a$, and the vertices of $G_2$, degree $b$. The construction goes similarly to that in the proof of Proposition~\ref{p:pgeb}. We take the first vertex of $G_1$ and join it to $a-p+1$ vertices $v_{q_1+1}, \dots, v_{q_1+a-p+1}$ of $G_2$ (with the understanding that $v_i=v_{i-q}$ when $i > q$). Note that $a-p+1 \le q$ (so we do not join to the same vertex of $G_2$ more than once) and that $a-p+1 \ge (b+2)-p+1 \ge 7$. The latter means that all the edges which we attach will belong to $3$-cycles in the resulting graph $G'$. Moreover, the same cycle $C_q$, with the same labelling of the vertices still has the property that the degrees of its vertices counted relative to $G'$ go similar to those in a top monotone cycle: the difference of the degrees is no more than one, and there is a path in $C_q$ starting at $v_1$ and passing through all the vertices of the maximal degree. We repeat the same procedure another $(p-1)$ times, joining each of the remaining vertices of $G_1=K_p$ one-by-one to $a-p+1$ vertices of $G_2$ which are consecutive along $C_q$ and which starting from the first vertex whose degree in the graph so far constructed is smaller than that of $v_1$ (or from $v_1$ if all such degrees are equal). The resulting graph $G$ realises the sequences $d$. Indeed, by construction, all the vertices belonging to $G_1$ will have degree $a$ in $G$. Moreover, the degrees in $G$ of the vertices belonging to $G_2$ differ by no more than one (this remains true at each step of our construction), there are $q$ of them, and the sum of their degrees is $bq$, which implies that all their degrees are equal to $b$. Furthermore, at each step, all attached edges belong to $3$-cycles. Also, $G_2$ is triangular by construction. Every edge of $G_1$ belongs to a $3$-cycle lying in $G_1$, with the only exception, when $p=2$. But then from~\eqref{eq:noKb+1} we have $a+b \ge q+1$, and so from $a \ge b+2$ we obtain $q \le 2a-3 < 2(a-p+1)$. This means that the two sets of $a-p+1$ edges which we add to two vertices of $G_1=K_2$ are simultaneously incident to at least one vertex of $G_2$, which then gives a $3$-cycle containing the edge $G_1=K_2$.
\end{proof}

In view of Propositions~\ref{p:sb1},~\ref{p:pgeb} and~\ref{p:pltb}, to complete the proof of the Theorem, it remains to show that all the sequences in the class $\mathcal{G}_\alpha, \, \alpha = 1,2,3$, admit a triangular realisation, where $\mathcal{G}_\alpha$ is the class of all biregular graphical sequences $d=(a^p, b^q)$ such that $p,q > 0, \, a \ge b+2, \, b \ge 4$, and $b = p+\alpha$.

The proof goes similarly for all three classes.

Suppose a sequence $d=(a^p, b^q) \in \mathcal{G}_1$ has no triangular realisation, and its degree sum is the smallest from among all sequences in $\mathcal{G}_1$ having no triangular realisation. We have $a \le p+q-1$ by \eqref{eq:EG2} and $a \ge b+2=p+3$, which imply $q \ge 4$. Performing two consecutive steps of the Havel-Hakimi reduction on the vertices of the smallest degree we obtain the (still graphical) sequence $d'=((a-2)^p, b^{q-2})$. If $a-2 \ge b+2$, then $d' \in \mathcal{G}_1$, and hence $d'$ admits a triangular realisation $G'$ as its degree sum is smaller than that of $d$. If $a-2 \in \{b,b+1\}$, then $d'$ has a triangular realisation $G'$ by Proposition~\ref{p:sb1}. We now construct a triangular realisation for $d$ by adding two vertices to $G'$, joining them by an edge, and then joining $p$ vertices of degree $a-2$ of $G'$ to each of them. Hence all the sequences in $\mathcal{G}_1$ have a triangular realisation.

Suppose a sequence $d=(a^p, b^q) \in \mathcal{G}_2$ has no triangular realisation, and its degree sum is the smallest from among all sequences in $\mathcal{G}_2$ having no triangular realisation. We have $a \le p+q-1$ by \eqref{eq:EG2} and $a \ge b+2=p+4$, which imply $q \ge 5$. Let $d'=((a-3)^p, b^{q-3})$. Unlike in the previous case, graphicality of $d'$ does not follow immediately (although it is easy to see that its degree sum is even). We consider several cases. First assume that $a-3 \ge b$. Then (by~\eqref{eq:EG2} and as $b > p$), graphicality of $d'$ is equivalent to the inequality $a-3 \le p+(q-3)-1$ which is satisfied, as $d$ is graphical. Furthermore, the sequence $d'$ either lies in $\mathcal{G}_2$ (and hence admits a triangular realisation $G'$ as its degree sum is smaller than that of $d$) or has its terms differing by no more than $1$ (and hence admits a triangular realisation $G'$ by Proposition~\ref{p:sb1}). Now we take a disjoint union of $G'$ with $K_3$ and then join $p$ vertices of degree $a-3$ of $G'$ to each of the vertices of $K_3$. The resulting graph $G$ is a triangular realisation of the sequence $d$. Now suppose that $a-3 < b$. As $a \ge b+2$, we must have $a=b+2$, and so $d=((p+4)^p,(p+2)^q)$ (with $p \ge 2, \, q \ge 5$ and $p+pq$ being even) and $d'=((p+2)^{q-3},(p+1)^p)$. By~\eqref{eq:EG2}, the sequence $d'$ is graphical if and only if $q \ge 6$. But if $q \le 5$, any realisation of $d$ is triangular by Remark~\ref{r:2b}. We can therefore assume that $q \ge 6$. If $p \ge 3$, the sequence $d'$ admits a triangular realisation $G'$ by Proposition~\ref{p:sb1}. If $q\ge 6$ and $p=2$, we obtain $d=(6^2,4^q)$ and $d'=(4^{q-3},3^2)$. This sequence has a triangular realisation $G'$ obtained by removing the edge $(0,2)$ from the standard triangular realisation of the sequence $(4^{q-1})$ constructed as in Proposition~\ref{p:sb1}. In both these cases, we can proceed as above to construct a triangular realisation of $d$.

Finally, let $d=(a^p, b^q) \in \mathcal{G}_3$ be a sequence with no triangular realisation whose degree sum is the smallest from among all sequences in $\mathcal{G}_3$ having no triangular realisation. From $a \le p+q-1$ and $a \ge b+2=p+5$, we obtain $q \ge 6$. Define $d'=((a-4)^p, b^{q-4})$. If $a-4 \ge b$, this sequence is graphical by~\eqref{eq:EG2} and admits a triangular realisation $G'$ (either because it lies in $\mathcal{G}_3$ and has a smaller degree sum than $d$ or by Proposition~\ref{p:sb1}). If $a-4=b-1$, we obtain $d=((p+6)^p, (p+3)^q)$ and $d'=((p+3)^{q-4}, (p+2)^p)$. As $d$ is graphical, we have $q \ge 7$ by~\eqref{eq:EG2}, and then $q \ge 8$ (because the degree sum of $d$ is odd when $q=7$). Then $d'$ is also graphical. Moreover, $p=b-3 \ge 1$ and we cannot have $p=1$, as otherwise the degree sum of $d$ is again odd. Then $p \ge 2$ and so $d'$ has a triangular realisation $G'$ by Proposition~\ref{p:sb1}. If $a-4=b-2$, we obtain $d=((p+5)^p, (p+3)^q)$ and $d'=((p+3)^{q-4}, (p+1)^p)$. As $d$ is graphical, we have $q \ge 6$ by~\eqref{eq:EG2}. From~\eqref{eq:EG2} applied to $d'$ (note that this time we need both inequalities!) we obtain that $d'$ is graphical if and only if $q \ge 8$ and $p^2-p(q-5)+(q-4)(q-8) \ge 0$. The left-hand side is positive for $q \ge 9$, and when $q=8$, it can only be negative when $p \in \{1, 2\}$. This gives two sequences $d^{(1)}=(6^1,4^8)$ and $d^{(2)}=(7^2,5^8)$ for which $d'$ is not graphical. Another two cases when $d'$ is not graphical are $q \in \{6,7\}$. But by Remark~\ref{r:2b}, any realisation of $d$ is triangular when $q \le p+5$, so we only get $(p,q)=(1,7)$ which gives $d^{(3)}=(6^1,4^7)$. All three sequences $d^{(1)}, d^{(2)}, d^{(3)}$ admit triangular realisations (see Remark~\ref{r:except} below), which contradicts the choice of $d$. It follows that we can assume that the sequence $d'=((p+3)^{q-4}, (p+1)^p)$ is graphical (and so $q \ge 8$). If $p \in \{1,2\}$ we get $d=(6^1,4^q)$ or $d=(7^2,5^q)$ which have triangular realisations by Remark~\ref{r:except}. Finally, let $p \ge 3$. Then the sequence $d'$ is graphical, with all the terms being at least $4$. If $d' \in \mathcal{G}_3$, it must have a triangular realisation $G'$ because its degree sum is smaller than the degree sum of $d$, and if $d' \notin \mathcal{G}_3$, it must again have a triangular realisation $G'$ because all the other cases in the proof are already considered. We now construct a triangular realisation of the sequence $d$ by taking a disjoint union of the graph $G'$ (in any of the above cases) and $K_4$, and then joining each of the $p$ vertices of degree $a-4$ of $G'$ to each of the vertices of $K_4$.

\begin{remark} \label{r:except}
  There is a small number of exceptional sequences not covered by general constructions in Propositions~\ref{p:sb1}, \ref{p:pgeb} and \ref{p:pltb}. We construct triangular realisations of these sequences explicitly.

  Let $d=(6^1,4^q)$. For $d$ to be graphical we need $q \ge 6$. By Remark~\ref{r:noKb+1} we can assume $q \le 10$. Performing the Havel-Hakimi reduction on the vertex of degree $6$ we obtain the sequences $d'=(4^{q-6},3^6)$. We now construct a realisation $G'$ of the sequence $d'$, for $6 \le q \le 10$, and then add to it a vertex of degree $6$ which we join to all six vertices of $G'$ of degree $3$. For $q=6$, we take $G'=K_{3,3}$. For $q=10$, we take for $G'$ the disjoint union of the graph $K_6$ minus two disjoint $2$-paths and the graph $K_4$. The graphs $G'$ for $q \in \{7,8,9\}$ are shown in Figure~\ref{figure:43}.

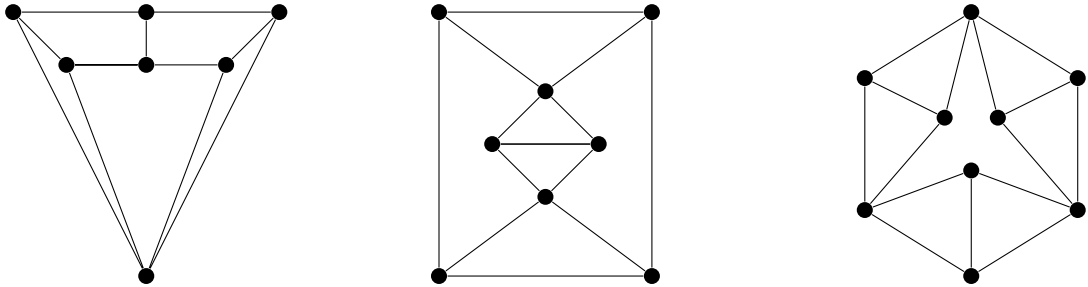
\begin{figure}[h]
\centering
\begin{tikzpicture}[-,scale=0.7]
  \tikzstyle{vertex}=[circle,fill=black,minimum size=6pt, inner sep=0pt]
  \def \x {2}
  \node[vertex] (G_1) at (0,0) {};
  \node[vertex] (G_2) at (5,0)   {};
  \node[vertex] (G_3) at (2.5,-5)  {};
  \node[vertex] (G_4) at (2.5,0) {};
  \node[vertex] (G_5) at (1,-1)   {};
  \node[vertex] (G_6) at (4,-1)  {};
  \node[vertex] (G_7) at (2.5,-1)  {};
  \draw (G_1) -- (G_2) -- (G_3) -- (G_1) -- (G_5) -- (G_7) -- (G_6) -- (G_3) -- (G_5) -- (G_7) -- (G_4);
  \draw (G_6) -- (G_2);
\node[vertex] (H_1) at (6+\x,0) {};
  \node[vertex] (H_2) at (10+\x,0)   {};
  \node[vertex] (H_3) at (10+\x,-5)  {};
  \node[vertex] (H_4) at (6+\x,-5) {};
  \node[vertex] (H_5) at (8+\x,-1.5) {};
  \node[vertex] (H_6) at (7+\x,-2.5)   {};
  \node[vertex] (H_7) at (9+\x,-2.5)  {};
  \node[vertex] (H_8) at (8+\x,-3.5) {};
  \draw (H_1) -- (H_2) -- (H_3) -- (H_4) -- (H_1) -- (H_5) -- (H_6) -- (H_7) -- (H_8) -- (H_3);
  \draw (H_4) -- (H_8) -- (H_6) -- (H_7) -- (H_5) -- (H_2);

\node[vertex] (I_1) at (14+2*\x,0) {};
  \node[vertex] (I_2) at (16+2*\x,-1.25)   {};
  \node[vertex] (I_3) at (16+2*\x,-3.75)  {};
  \node[vertex] (I_4) at (14+2*\x,-5) {};
  \node[vertex] (I_5) at (12+2*\x,-3.75) {};
  \node[vertex] (I_6) at (12+2*\x,-1.25) {};
  \node[vertex] (I_7) at (14+2*\x,-3) {};
  \node[vertex] (I_8) at (13.5+2*\x,-2) {};
  \node[vertex] (I_9) at (14.5+2*\x,-2)   {};
\draw (I_1) -- (I_2) -- (I_3) -- (I_4) -- (I_5) -- (I_6) -- (I_1);
\draw (I_3) -- (I_7) -- (I_5) -- (I_8) -- (I_1) -- (I_9) -- (I_3);
\draw (I_4) -- (I_7);
\draw (I_6) -- (I_8);
\draw (I_2) -- (I_9);
\end{tikzpicture}
\caption{Graphs $G'$ with degree sequences $d'=(4,3^6),\, (4^2,3^6),\, (4^3,3^6)$.}
\label{figure:43}
\end{figure}

  Let $d=(7^2,5^q)$. For $d$ to be graphical we need $q$ to be even and $q \ge 6$. By Remarks~\ref{r:noKb+1} and~\ref{r:2b} we can assume $8 \le q \le 11$, and so $q \in \{8,10\}$. If $q=8$, we start with the standard realisation $G'$ of the sequence $(4^9)$, with the set of vertices $\mathbb{Z}_9$ and with two vertices $i, j \in \mathbb{Z}_9$ joined by an edge if and only if $i-j \in \{\pm1, \pm 2\}$. Then we add to $G'$ two edges $(0, \pm 3)$, and an extra vertex (of degree $7$) which we join to all the vertices of $\mathbb{Z}_9$, except for the vertices $\pm 3$. It is easy to see that the resulting graph is a triangular realisation of $d$. Similarly, for $q=10$, we start with the circulant graph on the vertex set $\mathbb{Z}_{11}$ whose vertices $i, j \in \mathbb{Z}_{11}$ are joined by an edge if and only if $i-j \in \{\pm1, \pm 2\}$, add to it three edges $(0, \pm 3)$ and $(4,7)$, and then add an extra vertex of degree $7$ which we join to all the vertices of $\mathbb{Z}_{11}$, except for the vertices $\pm 3, 4$ and $7$. An easy inspection shows that the resulting graph is a triangular realisation of $d$.
\end{remark}

This completes the proof of the Theorem.

\medskip

For possible directions of solving the Conjecture, we note that reversing the Havel-Hakimi reduction (on either the first, or the second, or the last term of the sequence), ``in almost all the cases" preserves the following two simultaneous properties of the realisation: to be triangular and to contain a monotone Hamiltonian path (that is, a path passing through all the vertices along which the degrees do not increase). A modification of this argument may lead to an inductive proof; unfortunately, we were not able to find one.

\end{document}